\documentclass[11pt]{article}
\usepackage{amsmath}
\usepackage{amsthm}
\usepackage{amssymb}
\usepackage{verbatim}
\usepackage[dvips]{graphicx}
\usepackage{enumerate}
\usepackage{t1enc}
\usepackage[latin2]{inputenc}
\usepackage[english]{babel}
\usepackage{color}
\usepackage[margin=6mm,left=3cm,right=3cm,top=1cm,bottom=2cm,nohead]{geometry}
\newcommand{\R}{\mathcal{R}}

\theoremstyle{plain}
\newtheorem{theorem}{Theorem}
\newtheorem{lemma}{Lemma}
\newtheorem{proposition}{Proposition}
\newtheorem{corollary}{Corollary}

\theoremstyle{definition}
\newtheorem{definition}{Definition}

\newtheorem{question}{Question}
\newtheorem{conjecture}{Conjecture}
\theoremstyle{remark}

\theoremstyle{definition}

\title{On Path-Pairability in the Cartesian Product of Graphs}
\author{G\'abor M\'esz\'aros}
\begin{document}
\maketitle
\begin{abstract}
We study inheritance of path-pairability in the Cartesian product of graphs, and prove different (such as additive and multiplicative) inheritance patterns of path-pairability, depending on the size of the Cartesian product. We present path-pairable graph families, that improve the known upper bound on the minimal maximum degree of a path-pairable graph. Further results and open questions about path-pairability are also presented.
\end{abstract}

\section*{Introduction}
We discuss graph theoretic concepts, emerging from a practical networking problem introduced by Csaba, Faudree, Gy\'arf\'as, and Lehel in \cite{Cs}, \cite{mpp} and  \cite{pp}.
A graph $G$ on at least $2k$ vertices is called $k$-path-pairable if, for any pair of disjoint sets of (pairwise different) vertices $X=\{x_1,\dots,x_k\}$ and $Y=\{y_1,\dots,y_k\}$ of $G$, there exist $k$ edge-disjoint $x_iy_i$ paths joining the vertices. The path-parability number $pp(G)$ of a graph $G$ is the largest positive integer $k$, for which $G$ is $k$-path-pairable. A graph on exactly $2k$ vertices is simply called {\it path-pairable}, if it is $k$-path-pairable. The motivation of setting edge-disjoint paths between certain pairs of nodes naturally arose in the study of communication networks. There are various reasons to measure the capability of the network by its path-pairability number, that is, the maximum number of pairs of users, for which the network can provide separated communication channels without data collision. The inital problem and its graph theoretical model is discussed in \cite{Cs}.

Path-pairability is closely related to several other concepts, such as {\it linkedness} and {\it weak-linkedness}. A graph $G$ is $k$- linked/weakly $k$-linked if, for every ordered set of $2k$ vertices $X=\{x_1,\dots,x_k\}$ and $Y=\{y_1,\dots,y_k\}$, there exist vertex-disjoint/edge-disjoint paths $P_1,\dots, P_k$, such that each $P_i$ is an $s_it_i$-path. We wish to highlight, that, while the definition of weak linkedness may resemble our earlier definition of path-pairability, repetition of the vertices in the terminal list is allowed for weak-linkedness, and it is forbidden in case of path-pairability. Note that in case of linkedness, the two conventions lead to equivalent concepts. By definition, weakly $k$-linked graphs are $k$-path-pairable,
thus path-pairability is a special variant of weak-linkedness. Nevertheless, the two properties differ in several respects. Weakly $k$-linked graphs are necessarily $k$-edge-connected, while $k$-path pairable graphs only must satisfy a milder, so called {\it cut - condition}.

\begin{definition}[Cut-condition] A graph $G$ satisfies the {\it $k$-cut-condition} if, for every  $S\subset V(G)$ where $|S|\leq k$, $d(S)\geq |S|$ holds. A graph $G$ on $2n$ vertices satisfies the {\it cut-condition }, if for every  $S\subset V(G)$, $|S|\leq n$, $d(S)\geq |S|$ holds. 
\end{definition}

If $G$ is $k$-path-pairable, then it satisfies the $k$-cut condition. Indeed, if there exist $S\subset V(G)$ that violates the condition, terminals placed at every vertex of $S$, with their pairs in $G\backslash S$ 
 cannot be joined without subsequent use of at least one edge between the two sets.
  Note that the cut condition states a necessary but not sufficient condition for path-pairability. On the other hand, $k$-path-pairable graphs do not even have to be $2$-edge-connected. The star graph $K_{1,n}$ is one of the most illustrative counterexamples, being connected (but not 2-edge-connected) and $\lfloor\frac{n}{2}\rfloor$ path-pairable. Faudree, Gy\'arf\'as, and Lehel \cite{3reg} gave examples of $k$-path-pairable graphs with maximum degree $\Delta = 3$, for arbitrary values of $k$. In contrast, the same authors proved \cite{mpp}, that the maximum degree has to grow together with the graph size in path-pairable graphs. They in fact showed, that a path-pairable graph with maximum degree $\Delta$ has at most $2\Delta^\Delta$ vertices. The result places a lower bound of $O(\frac{\log n}{\log\log n})$ on the maximum degree of a path-pairable graph on $n$ vertices. This bound is conjectured to be asymptotically sharp, though examples of path-pairable graphs with maximum degree of the right order of magnitude have yet to be explored. The best known constructions are due to Kubicka, Kubicki and Lehel \cite{grid} as well as M\'esz\'aros \cite{me_diam} and have maximum degree of order of magnitude $O(\sqrt{n})$. The construction in \cite{grid} is obtained by taking the Cartesian product of two complete graphs. That motivated the author of this present paper to study  path-pairability in the Cartesian product in more details.
 
The {\it Cartesian product} of graphs $G$ and $H$ is the graph $G\square H$ with vertices $V(G\square H)=V(G)\times V(H)$, and $(x,u)(y,v)$ is an edge, if $x=y$ and $uv\in E(H)$ or $xy\in E(G)$ and $u=v$.
The Cartesian product of graphs has been extensively studied in the past decades. It gave rise to important classes of graphs; for example, the $n$-dimensional grid can be considered as the Cartesian product of lower dimensional grids. Hypercubes are well known members of this family with similar recursive structures:  the Cartesian product of $m$-dimensional and $n$-dimensional hypercubes is an $(m+n)$-dimensional one.  
The study of graph products leads to various deep structural problems such as invariance and inheritance of graph parameters.  We mention a couple of relevant results within the field of linkedness and its variants, with no claim of being exhaustive. 
Chiue and Shieh \cite{ChiueShieh} proved, that the Cartesian product of a $k$-connected and an $l$-connected graph is $(k+l)$-connected. Similar result for edge connectivity was proved by Xu and Yang \cite{XuYang}. Inheritance of linkedness has been investigated by M\'esz\'aros \cite{me_a+b}, who proved that the Cartesian product of an $a$-linked graph $G$ and a $b$-linked graph $H$ is $(a+b-1)$-linked, given that the graphs are sufficiently large in terms of $a$ and $b$. 

This paper has two main objectives. We prove an inheritance theorem of path-pairability (Theorem \ref{pp a+b}), that is similar to the inheritance of linkedness, presented in \cite{me_a+b}. We also prove an extension of Theorem \ref{pp a+b}, which states that, given sufficient space in the product graph, reasonably higher path-pairability can be achieved (Theorem \ref{pp oom}). We mention that neither linkedness, nor weak-linkedness share this property. 
 
\begin{theorem}\label{pp a+b}
If $G$ is an $a$-path-pairable graph with $|V(G)|\geq 8a$ and $H$ is a $b$-path-pairable graph with $|V(H)|\geq 8b$, then $ G\square H$ is $(a+b)$-path-pairable.
\end{theorem}
\begin{theorem}\label{pp oom}
If $G$ is an $a$-path-pairable graph and $H$ is a $b$-path-pairable graph and $v(G), v(H)\geq 4s$, $s < \frac{(a+1)(b+1)}{2}$, then $ G\square H$ is $s$-path-pairable.
\end{theorem}
\begin{corollary}\label{pp ab}
If $G$ is an $a$-path-pairable graph and $H$ is a $b$-path-pairable graph and $v(G), v(H)\geq 4\cdot \frac{(a+1)\cdot (b+1)}{2}-1$, then $ G\square H$ is $\big(\frac{(a+1)\cdot (b+1)}{2}-1\big)$-path-pairable.
\end{corollary}

Theorem \ref{pp a+b} and \ref{pp oom} concern themselves with path-pairability of the product graph $G\square H$, where path-pairabilites of the factors $G$ and $H$ are conveniently small, compared to the graph sizes $|V(G)|$ and $|V(H)|$. Our other objective is the examination toward the other extremity, when $pp(G)$ and $pp(H)$ are as large as possible, that is, both $G$ and $H$ are path-pairable. The ultimate goal would be to find sufficient conditions that guarantee path-pairability of the product graph, thus offer a valuable tool to generate new path-pairable graph families. To date, very little is known about this kind of inheritance.
The Cartesian product of two path-pairable graphs is not necessarily path-pairable. A counter-example is presented in Proposition \ref{stars}. On the other hand, it is still an open and quite annoying question, if path-pairability of at least one of the multiplicands is necessary at all for path-pairability of the product graph. We believe, that the described condition is not necessary, but cannot verify it by means of a counterexample, hence we state it as a conjecture.
\begin{conjecture}
There exist non-path-pairable graphs $G$ and $H$, such that $G\square H$ is path-pairable.
\end{conjecture} 

Kubicka, Kubicki, and Lehel \cite{grid} investigated path-pairability of complete grid graphs, that is, the Cartesian product of complete graphs, and proved that the two-dimensional complete grid $K_a\times K_{b}$ of size $n=ab$ is path-pairable. Our objective is to improve the presented result and show, that the Cartesian product of the complete bipartite graph $K_{m,m}$ with itself is path-pairable for sufficiently large even values of $m$. The examined path-pairable product has $n=4m^2$ vertices and maximum degree $\Delta = 2m=\sqrt{n}$, which improves the previously discussed upper bound ($\approx \sqrt{2}\sqrt{n}$) on $\Delta(G)$ to $\sqrt{n}$. It also presents a new infinite family of path-pairable graphs, as well as gives examples of non-complete path-pairable graphs, whose Cartesian product is path-pairable as well.
\begin{theorem}\label{m}
The product graph $K_{m,m}\square K_{m,m}$ is path-pairable for even values of $m$, if $m\geq 104$.
\end{theorem}
We follow the notation of \cite{product}. For the sake of completeness, we recall definitions of the mainly used concepts. A $G$-layer $G_x$ ($x\in V(H)$) of the Cartesian product $G\square H$ is the subgraph induced by the set of vertices $\{(u,x): u\in V(G)\}$.  An $H$-layer is defined analogously. We call edges of $G\square H$ lying in $G$-layers vertical while edges lying in $H$-layers are called horizontal. Unless it is misleading, we also use the notation $G_z=G_x$ and $H_z=H_y$ for layers corresponding to $z=(x,y)\in G\square H$. 

 We also refer the reader to \cite{product} for further details on product graphs. For a comprehensive survey of results concerning path-pairability, we refer to \cite{F3} and \cite{pp}.

\section*{Proof of Theorem \ref{pp a+b} and Theorem \ref{pp oom}}



Let $M$ denote the set of $2(a+b)$ (arbitrarily chosen and paired) terminals in $G\square H$. We may assume that $a\geq b$. We first prove the theorem in the "base" case, when no $G$-layer contains terminals belonging to $(a+1)$ or more pairs. The assumption in fact implies that no layer contains more than $2a$ terminals.
Our goal is to join terminals lying on the same $G$-layer, while choosing a "pseudopair" $u'$ of every remaining terminal $u$ such that $u'\in G_u$. Similarly, if $v$ denotes the real terminal pair of $u$ (such that $v\not\in G_u$), we choose $v'$, such that $u'$ and $v'$ are on the same $H$-layer. We join $uu'$, $u'v'$, and $v'v$ pairs for all $(u,v)$ terminal pairs by edge-disjoint paths. The union of such path-triples will join the initially set terminal pairs. We describe the above steps in more details as follows.

  Take a $G$-layer $G_x$ $(x\in H)$ with terminals $u_1,\dots,u_t$ ($1\leq t \leq 2a$). Observe that, if $t=a+s$ where $1\leq s\leq a$, then $G_x$ contains at least $s$ pairs of terminals. For an unmatched terminal $u$ of $G_x$, we choose a pseudo pair $u'\in G_x$, such that different terminals get different pseudo pairs and $H_{u'}$ contains no other terminal, but it contains the pseudo-pair of $v$, the terminal pair of $u$. Since $|V(G)|\geq 8a$, we can freely assign terminal-free vertical layers for the pseudo pairs of each pair of the terminals. Moreover, this assignment can be carried out even with the additional constraint, that no vertical layer will contain more than $b$ pairs of pseudo pairs. Now every $G$-layer contains at most $a$ pairs of terminals or terminal-pseudo-pair pairs. Using the fact the $G$-layers are $a$-path-pairable, we can assign edge-disjoint paths joining the pairs within every one of the layers. Having done that, the appropriate $(u',v')$ pseudopairs can paired within their $H$-layers by an arbitrary path. That completes the proof of the base case.

We mention that our presented technique wastes a lots of potential in the pairing of the pseudopairs. Using that $H$, and so every $H$-layer is $b$-path-pairable, $\frac{2a+2b}{2b}\leq a$ additional empty $H$-layers are sufficient to finish the pairing, hence fewer $H$-layers suffice to contain the pesudopairs. The lower bounds on the graph sizes in the theorem can be improved to $|V(G)|\geq 5a$ and $|V(H)|\geq 5b$ in the discussed case. We continue our proof with the initial weaker bounds. 

Now we turn to examination of the general case. As $4(a+1)>2(a+b)$, at most 3 $G$-layers contain $(a+1)$ or more types of terminals. Our goal is to reduce our problem to the base case, by redistributing the terminals among the $G$-layers. It will be done by assigning pseudopairs for each terminal within its original $H$-layer. Observe that, as the solution of the base case contains a horizontal shift, the combination of the initial redistribution, and the solution of the base case will use no vertical or horizontal edge more than once. For the redistribution of the terminals, we  follow a case-by-case analysis. 

\begin{enumerate}
\item Assume first, that $G_x$ is the only $G$-layer that contains $u_1,\dots, u_{a+t}$ terminals belonging to different pairs, where $1\leq t\leq b$. It means there are at most $(a+2b-t)$ terminals outside of $G_x$. We claim that one of these layers contains at most $(a-t)$ terminals, else the graph $G\square (H-x)$ would contain at least $(8b-1)(a-t+1) > (a+2b-t)$ terminals, clearly contradicting our previous observation. Take a $G$-layer $G_y$ with the above property. We want to choose $t$ of the terminals in $G_x$ (if their pair is in $G_x$ as well, then we choose both of them) and assign them pseudopairs in $G_y$, together with vertical paths joining the terminals to their pseudopairs. Note that we cannot assign a pseudopair to a vertex that already contains a terminal.
The terminals initially placed in $G_y$ prohibit the assigment of pseudo pairs for at most $a$ of the terminals (singleton or paired) of $G_x$, that is, at least $(a+t)-(a-t)=2t$ terminals can get pseudopairs, while we only needed $t$. Note also that the total number of types of terminals and pseudopairs in $G_y$ is at most $(a-t)+t =a$ after the redistributing step, as prescribed in the base case. We can now apply the solution of the base case on a new set of terminals, where pseudopairs take the place of their initial terminals. 

\item If two $G$-layers contain at least $(a+1)$ types of terminals, the remaining terminals occupy at most $(2b-2)$ $G$-layers, that is, there exists at least $6b$ $G$-layers that are free of terminals. 
If $b=1$, both layers contain exatly $(a+1)$ terminals of different pairs. One can arbitrarily pick a pair, shift them vertically just as in the previous case, completing our task. If $b\geq 2$, every $H$ is at least 2-path-pairable, hence we can arbitrarily define for a terminal $u$ a pseudopair $u'$ in $H_u$, such that

\begin{enumerate}
\item $G_{u'}$ contains no terminal and contains at most $a$ pseudopairs at the end of the procedure,
\item $uu'$ pairs are joined within $H_u=H_{u'}$ by edge-disjoint horizontal paths.
\end{enumerate} 
Indeed, to satisfy the first condition, observe that we have at most $(2a+2b)$ terminals that we distribute among $6b$ empty layers without any particular constraint (remember, here a terminal and its pair do not have to get pseudopairs assigned to the same $G$-layer), thus a balanced distribution with at most $\lceil \frac{2a}{3b} \rceil\leq a$ terminals can be chosen. The second condition can be guaranteeed by 2-path-pairability, as we assign at most 2 pseudopairs within an $H$-layer.

\item The case with three overloaded layers ($G_1$, $G_2$, $G_3$) works similarly to the previous one. Observe first that in the examined case $3(a+1)\leq (2a+2b)$, hence $b\geq \frac{a+3}{2}\geq 2$. Remember, that $a\geq b$, thus $a\geq \frac{a+3}{2}\Rightarrow a\geq 3$, which yields $b\geq 3$ as well. By pigeon-hole principal, we have at least $(a+6b)$ empty $G$-layers at disposal, each of them expected to receive $\lceil\frac{2a+2b}{a+6b}\rceil\leq a$ pseudopairs on average. Since $b\geq 3$, the at most three paths can be established within every $H$-layer, that completes the examination of the case and so the proof as well. 
\end{enumerate}

Before proving Theorem \ref{pp oom}, we state that the bound of Theorem \ref{pp a+b} gives the right order of magnitude of path-pairability for certain classes of graphs.

\begin{proposition}\label{stars}
The Cartesian product $K_{1,b}\square K_{1,d}$ is at most $ \lceil\frac{b+d}{2}\rceil$-path-pairable.
\end{proposition}
\begin{proof}
Let $C$ and $R$ denote the sets of vertices of degree two in an arbitrary column and an arbitrary row not contanining the unique vertex of degree $(a+b)$ (denoted by $z_{a+b}$) and let $x$ be an additional vertex of degree two. We denote the unique vertex of the intersection $C\cap R$ by $y$.  We place terminals in $C\cup\R\cup\{x\}$ such that $x$ and $y$ form a pair, as well as the unique vertices of degree $(a+1)$ and $(b+1)$ (denoted by $z_{a+1}$ and $z_{b+1}$) form another. Observe that paths that join the above two pairs both use either the edge between $z_{a+1}$ and $z_{a+b}$ or between $z_{b+1}$ and $z_{a+b}$, hence the pairing cannot be achieved.
\end{proof}

We believe that, with a somewhat longer and more cumbersome analysis, one can actually prove that $pp(K_{1,b}\square K_{1,d})=\lceil\frac{b+d}{2}\rceil$. Nevertheless, the actual proof is bound to require a lengthy case-by-case analysis, that we do not consider particularly interesting and do not investigate. 

Now we turn to the proof of Theorem \ref{pp oom}. We use the same techniques as in the previous proof. Again, we may assume $b\leq a$. If no $G$-layer contains more than $a$ different types of terminals, we can join the pairs that share a $G$-layer, and assign pseudopairs to the terminals having their pairs on a different $G$-layer, just as we did in the base case of the previous proof. The pseudopairs can be chosen, such that
\begin{itemize}
\item[i)] their $H$-layers contain no terminal,
\item[ii)] pseduopairs of a pair of terminals are located on the same $H$-layer, and
\item[iii)] every $H$-layer contains at most $b$ pairs of pseudopairs.
\end{itemize}
The initial terminals occupy at most $2s$ $H$-layers. We need and additional empty $H$-layer for every one of the $s$ pairs, that is guaranteed by the condition $|V(G)|\geq 4s$. Pairing of the pseudopairs can be carried out within the $H$-layers Again, we are far from an optimal solution, as an $H$-layer is capable of joining up to $b$ pairs of terminals, hence similar theorem with a stronger condition on the graphs size can be proved.  For the sake of convenience and clarity, we stick to the weaker variant and proceed by investigating the general case.

If $G_{x_1},\dots,G_{x_t}$-layers contain more than $a$-types of terminals, observe first that $t\leq b$, else $G\square H$ would consist of at least $(a+1)(b+1)$ terminals, contradicting $s < \frac{(a+1)(b+1)}{2}$ . It means, that in every vertical layer that contains a terminal $u$, we can assign a pseudopair $u'$ and \-- using that $H$ is $b$-path-pairable, and so is every vertical layer in $G\square H$\-- define edge disjoint $uu'$ paths for every $u$. We can distribute the pseudopairs among the initially empty horizontal layers equally, such that none of them contains more than $a$ pseudopairs. Indeed, we have at least $2s$ empty $G$-layers at our disposal and have to redistribute $2s$ terminals in total. Having done this, we can join the pseudopairs as described in the above base case.

Corollary \ref{pp ab} follows trivially from Theorem \ref{pp oom}. We show, that the bound presented in Corollary \ref{pp ab} is also sharp, up to a constant factor. That is, the order magnitude in the inheritance of path-pairability cannot be expanded more than indicated in Theorem \ref{pp oom}, by simply providing an abudance of space in the product graph. To prove our claim, we first make the following observation: if $G_0\subset G$ and $H_0\subset H$ subsets violate the cut-condition, that is, $e(G_0)<|G_0|$ and $e(H_0)<|H_0|$, the product set $G_0\square H_0$ does not necessarily have the same condition. In order to generate violating product sets, stronger assumptions are needed:
\begin{proposition}\label{doubleviolate}
Let $G$ and $H$ be graphs and $G_0\subset G$ and $H_0\subset H$, such that $2\cdot e(G_0)<|G_0|$ and $2\cdot e(H_0)<|H_0|$. Then $e(G_0\square H_0)< |G_0\square H_0|$, that is, $G_0\square H_0$ violates the cut condition.
\end{proposition}
\begin{proof}
Clearly $|G_0\square H_0|= |G_0|\cdot |H_0|$, while $e(G_0\square H_0)=|G_0|\cdot e(H_0)+|H_0|\cdot e(G_0)<\frac{|G_0|\cdot |H_0|}{2}+\frac{|G_0|\cdot |H_0|}{2}=|G_0|\cdot |H_0|$.
\end{proof}

We construct our example by a graph operation called "blowing-up".
 Let $n=k\cdot m$, and define $G(k,m)$ as an equally blown up graph of the path $P_{k}$ of size $n$, that is, $V(G)=\{x_{i,j}: 0\leq i\leq (k-1), 0\leq j\leq (m-1)\}$, where $x_{i,j}$ and $x_{i',j'}$ are connected, if either $i=i'$, or $i-i'=\pm 1$ (modulo $2m$). In other words, we take a path on $k$ vertices, replace every vertex by a complete graph $K_m$, and every edge of the initial path by the edge set of a complete bipartite graph $K_{m,m}$ between the two cliques. We  use the notation $S_i=\{x_{i,j}\in V(G): 0\leq j\leq (k-1)\}$ and refer to the set as the $i$th class of $G$.
\begin{proposition}
$G(k,m)$ is $m^2$-path-pairable, if $k\geq 2m$.
\end{proposition}
\begin{proof}
 Given a distribution of $m^2$ pairs of vertices, we can carry out pairing by starting at one end of the path, greedily joining terminals to vertices of the consecutive class, and finishing the joining of terminals within the classes. For a terminal $u$, we will assign several $u',u'',\dots$ pseudopairs in the consecutive classes until we finally pair one with the appropriate $v$ pair. We start by pairing terminals that lie in the same class by direct edges of the cliques. From now on we may assume that, for every pair $(u,v)$, one of the terminals is closer to the left end of the path, hence it will be encountered earlier in our left-to-right sweeping algorithm than its pair. Being at class $S_{i}$, the consecutive class $S_{i+1}$ contains at most $m$ terminals. If some of them have appropriate pseudopair in $S_i$, they can be joined by direct edges (here we are using that path-pairability prohibits repeated terminal assignment of a vertex). Then, the remaining terminals of $S_i$ can be assigned a new pseudopair in $S_{i+1}$, maintaining the condition that a vertex $x\in C_{i+1}$ hosts at most $m$ terminals and pseudopairs that have not been paired. Having visited at most $t^2$ terminals, this condition can be easily maintained using Hall's Matching Theorem. Having reached $t^2+a$ terminals, we must have encountered at least a pairs, that is, the number of still unmatched terminals is at most $(t^2-a)$, thus our above reasoning works just as well as before.
\end{proof}

 Now let $G=G(a,k)$ and $H=G(b,k)$, such that $a\geq b\geq 2$ and $k\geq (4a^2+1)$. Moreover, let $G_0\subset G$ and $H_0\subset H$ be formed by $(2a^2+1)$ and $(2b^2+1)$ consecutive classes, starting at the left end of the blown-up paths. The sets $G_0$ and $H_0$ satisfy the conditions of Proposition \ref{doubleviolate}, thus $G\square H$ is not $(2a^2+1)\cdot(2b^2+1)$-path-pairable, regardless of the initial sizes of $G$ and $H$. That justifies our claim.

\section*{Proof of Theorem \ref{m}}

Let us denote the two classes of the bipartite graph $K_{m,m}$ by $A_1$ and $A_2$. We introduce further notation for certain sets of the vertices in the product graph $G=K_{m,m}\square K_{m,m}$ as follows:
$A_{11}=A_1\square A_1$, $A_{12}=A_1\square A_2$, $A_{21}=A_2\square A_1$, and $A_{22}=A_2\square A_2$. We will refer to these sets as {\it classes} of $G$. We set a cyclic order of the four classes clockwise. References to {\it next class} and {\it previous class} are translated in accordance with that given cyclic order. We label the $m^2$ elements of each class by $(u,v)$ pairs, where $u=1,\dots,m$ and $u=1,\dots,m$. We introduce the expression of {\it shipping} a terminal or pseudopair $u$ to a vertex or pseudopair $v$, by which we mean that we establish an $uv$ path $P_{uv}$ between the two vertices, such that $P_{uv}$ shares no edge with any other path. We will join our terminals by shipping them several times, that is, taking the union of several paths defined between appropriate sequences of pseudopairs. Having read the proof of Theorem \ref{pp a+b} and Theorem \ref{pp oom}, this method is likely to look familiar for the reader. A vertex is said to {\it host} a terminal, if the terminal is shipped to the vertex at some point during our pairing procedure.

Given a pairing of the vertices, we carry out the joining of the terminals in three phases named: {\it swarming}, {\it line-up} and {\it final match}. For a pair of terminals of $G$, we first ship them to the same class (swarming), then send them forward to the same row/column of the next class (line-up). Finally, we join the to paths by their newly established ends with a single vertex of the next class (final match).

\subsection*{Swarming} In this phase, we ship one terminal of each pair to the class of its partner. If a pair lies with both vertices within a class, they simply skip the swarming phase.
A terminal $(u,v)$, belonging to class $A_{11}$  and heading to $A_{12}$, shall follow the path $(u,v)\rightarrow (u+1,v)$,  where $(u+1,v)$ denotes the appropriate vertex of $A_{12}$ and addition is calculated modulo $m$.
Similarly, we ship $(u,v)$  to $A_{21}$ via the path  $(u,v)\rightarrow (u,v+1)$. Should $(u,v)$ be shipped to $A_{22}$, we allocate it the path $(u,v)\rightarrow (u+1,v)\rightarrow (u+1,v+2)$ where $(u+1,v)$ belongs to $A_{12}$ and $(u+1,v+2)$ belongs to $A_{22}$.
Terminals belonging to other classes will be shipped by the same rules, increasing the appropriate coordinate by 1 at the first step, and increasing the other one by 2 in the second step, if applicable. Getting shipped via paths of length two is always carried out clockwise.

One can easily verify, that the above arrangment of paths assures that, if $m\geq 5$, no edge is being utilized twice during the swarming phase. We now choose the terminal to be shipped for each pair, such that at the end of the swarming phase, every class hosts exactly $\frac{m^2}{2}$ pairs. Starting with an arbitrary selection, we can assume without loss of generality, that $A_{11}$ hosts the most pairs, and that at least one terminal $x\in A_{11}$ received its pair $y$ from a class hosting less than $\frac{m^2}{2}$ pairs. Shipping $x$ to the class of $y$ instead balances the distribution of the pairs. Repetition of the previous step leads to an equal distribution. 

We define $G'$ with $V(G')=V(G)$, and a new edge set $E(G')$ by deleting those edges from $E(G)$ we used in the swarming phase. Observe, that by the given shipping method, every vertex of $G$ hosts at most 5 terminals and uses at most 8 of its edges, that is, the minimal degree of $G'$ is at least $(m-8)$.  We continue the linking in $G'$.  
\subsection*{Line-up}  We ship each pair of terminals to the next class, such that terminals
 shipped by a horizontal edge shall share the same column of the new class, while vertically shipped terminals will arrive in the same row. For every pair, there are at least $(m-16)$ available columns/rows in the next class. Our intention is to pair up the pairs with the rows/columns, such that every one of them will contain $\frac{m}{2}$ pairs. We recall a straigthforward corollary of Hall's Matching Theorem.
\begin{lemma}\label{Hall}
A bipartite graph $G=(A,B,E)$ with vertex classes of size $n$
whose minimum degree is at least $\frac{n}{2}$ contains a perfect matching.
\end{lemma}
We define the following bipartite graph $G=(A,B,E)$ as follows: represent each pair of terminals hosted in $A_{11}$ by a vertex in $A$, while each column of $A_{12}$ is represented by $\frac{m}{2}$ independent vertices in $B$. Certainly, $|A|=|B|=\frac{m^2}{2}$. We connect two vertices of $A$ and $B$ by an edge, if both terminals of the corresponding pair have horizontal edges to the corresponding column of $A_{12}$. Easy to see, that the graph has minimum degree at least $\frac{m^2}{2}-16m$, hence, by Lemma \ref{Hall}, it contains a perfect matching for $n\geq 64$.

Observe, that if two pairs of terminals sharing a vertex of a class $C$ are distributed to the same vertical layer of the next class $C'$, at least one of the terminals will not be able to get shipped there. We need to guarantee a matching between the pairs and the layers of $C'$ without such a collision. Recall, that each vertex of $C$ hosts at most 5 terminals, hence each pair of terminals has at most 8 additional pairs to collide with. Consider a perfect matching for which the number of above collisions is minimal. Let $(x,y)$ and $(x',y')$ colliding pairs of terminals being sent to layer $L$ of $C'$. We may assume $x$ and $x'$ share the same vertex of $C$. We want to find a pair $(u,v)$ sent to a layer $L'\neq L$ of $E$ such that
\begin{enumerate}
\item[i)] $(x,y)$ can be sent from $C$ to $L'$ (instead of $L$) during the line-up without causing further collision,
\item[ii)] $(u,v)$ can be sent from $C$ to $L$ (instead of $L'$) during the line-up without causing further collision.
\end{enumerate}
The pair $(x,y)$ can be initially sent to $(m-16)$ layers of $C'$, at most 8 of which might contains teminals that initially shared vertex with $(x,y)$ in $C$. In order to avoid further collisions we exclude these layers, leaving us at least $(m-24)$ choices of $L'$. We also want to exclude layers that alreay received terminals from the vertex of $x$ or $y$, yielding at most 8 additional excluded layers, that is, at least $(m-32)$ choices of $L'$ and so $(m-32)\cdot\frac{m}{2}$ choices for $(u,v)$.
We want to choose $(u,v)$ such that it initially did not share vertex in $C$ with any terminal currently hosted in $L$ and that $u$ and $v$ still can be moved (having witdrawn from $L'$) from $C$ to $L$ (the corresponding edges have not been used yet). 
For the first constraint, recall that $L$ contains $\frac{m}{2}$ pairs, every one of which shares vertex with at most 8 additional terminals. There are at most $4m$ additional terminals that initially cannot be sent to $L$, because the appropriate edges had already been used during the first phase.

Now assume that the appropriate edge, that would channel $u$ or $v$ to $L$ has already been used. It can either occur if another terminal was sent from that particular vertex of $C$ to $L$ during the line-up, or if the edges were used during the swarming phase. The first conditions means, that $(u,v)$ collides with the other pair of terminals that was sent to $L$, hence $(u,v)$ is one of the above listed $4m$ pairs. In the remaining case, the missing edge is one of those at most $8\cdot\frac{n}{2}=4m$ edges the complete layer $L$ used up during the swarming. The mentioned edges have at most $4m$ endpoints in $C$ and at most $5\cdot4m=20m$ pairs of terminals correspoding to them. 

 Overall, it means that if $(m-32)\cdot\frac{m}{2}>24m $ (that is, $m>56$) , one can find an appropriate $(u,v)$. Swaping the positions of $(u,v)$ and $(x,y)$, we reduced the number of collisions, contradicting our assumption. 

We repeat the same procedure for the remaining three classes. It can be easily verified that no edge is used more than once.  We define $G''$ by the deletion of the used edges the same way we obtained $G'$. We proceed in $G''$ to the final match. 

\subsection*{Final match} 
For a row/column filled with $\frac{m}{2}$ pairs of terminals, we assign every pair a vertex of the appropriate row/column of the next class, being adjacent to both terminals (see Figure \ref{lineup}). Note that during the first two phases, each vertex has used at most $13$ of its edges. We use Lemma \ref{Hall} to find the appropriate assigment. Let $A$ form the set, in which every pair of terminals of a certain row/column is represented by a vertex. The set $B$ is formed by any $\frac{m}{2}$ vertices of the appropriate column/row of the next class. We connect vertices by edges, if both terminals of the pair are adjacent to the appropriate vertex in the next class. Our bipartite graph has two classes of size $\frac{m}{2}$ and minimum degree $\frac{m}{2}-26$. If $m\geq 104$, the required matching is provided by Lemma \ref{Hall}. That completes the proof.  
\begin{figure}[h]\label{lineup}
\begin{center}
\includegraphics[width=5cm]{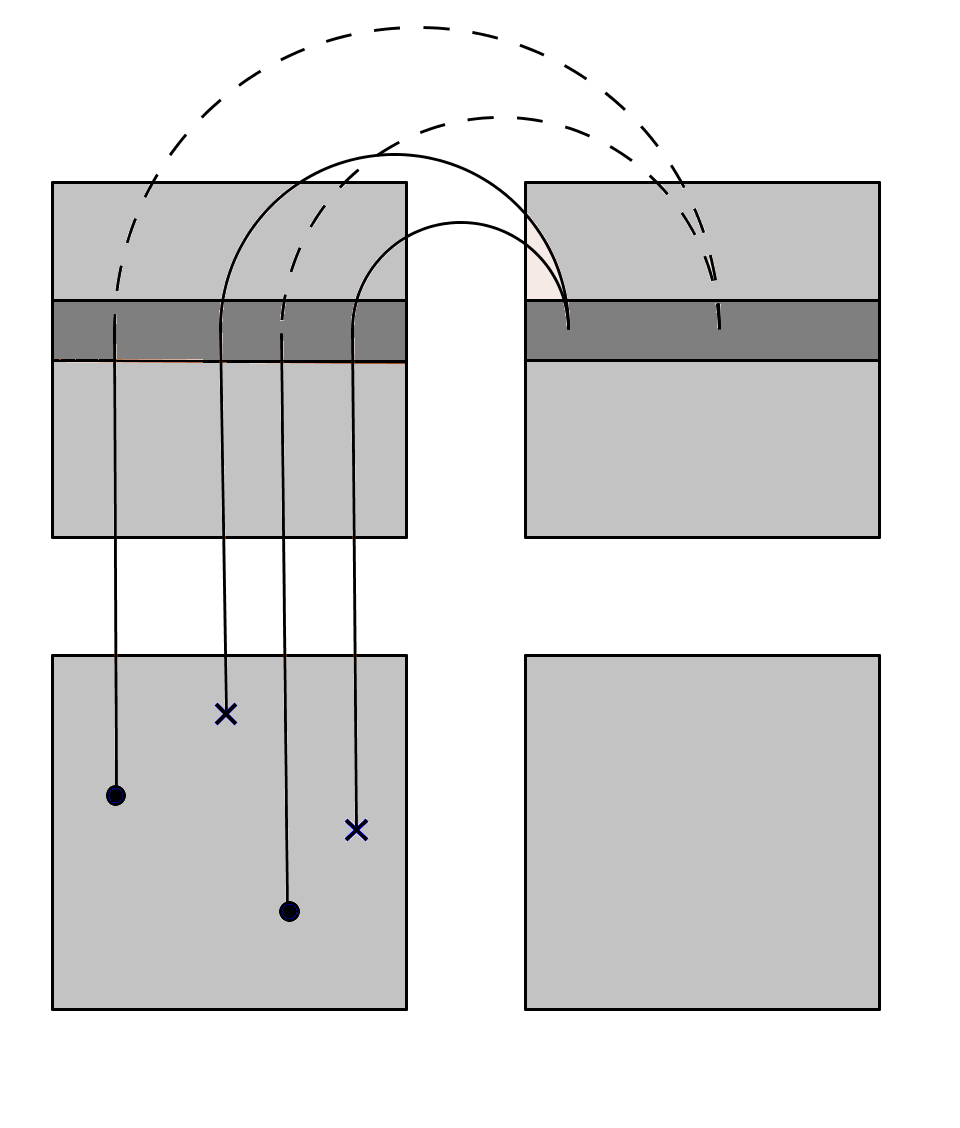} \hspace{1cm}
\caption{Line-up and final match phases.}
\end{center}
\end{figure}
\begin{corollary} 
There exists a path pairable graph $G$ on $n$ vertices with $\Delta(G)=\sqrt{n}$ for infinitely many values of $n$. 
\end{corollary}

\section*{Additional remarks and open questions}

\subsection*{The cut-condition is not sufficient}
We first prove, that the $k$-cut condition does not imply $k$-path-pairability. Consider the disjoint union of the star grah $K_{1,k}$ and the complete graph $K_N$ on $N\geq 2k$ vertices. Join each vertex of degree one to an arbitrary vertex of $K_N$ by an edge, such that different vertices of $K_{1,k}$ are joined to different vertices of $K_N$. The graph $G$ obtained this way is clearly not $k$-path pairable. Indeed, placing a $k+1$ terminals in $K_{1,k}$, such that the pair of the terminal in the center of the star graph is placed in $K_N$, any path starting in the center of the star severs its neighbor, devouring both of its edges. On the other hand, if $S\subset V(G)$ it trivially satisfies the cut-conditions, if it contains any vertex of $K_N$. Hence we may assume $S\subset K_{1,k}$, which case the verification of the cut-condition is straightforward.

Appropriate fine-tuning of the construction provides examples of graphs that are not path-pairable, while they satisfy the cut-condition. Take the disjoint union of $K_{1,k}$ and $K_{k-1}$ and join the two graphs by a matching (avoiding the center of the star) of size $k-1$. Join the remaining vertex of degree one to any vertex of $K_{n-1}$. Just as before, the set of degree-two vertices joined to the center of the star make the graph impossible to channel $k$ edge-disjoint paths. We claim that our graph $G$ satisfies the cut-condition for $k\geq 6$. 
Assume on the contrary that $S\subset V(G)$ of size at most $k$ violates the condition. We proceed by case-by-case analysis.
\begin{description}
\item[Case 1] If $K_{k-1}\subset S$, $S$ must contain an additional vertex, that is, $|S|=k$. Easy to see, that adding neither the center nor any end of the star graphs to the vertex set of $K_{k-1}$ violates the condition.
\item[Case 2] If $|S\cap K_{k-1}|=k-2$, $d(S)\geq k-2$ because of the edges leaving $S$ within $K_{k-1}$. Also,  at least $k-4$ of them have a neighbor in $K_{1,k}$ not belonging to $S$, that is, $d(S)\geq k-2+k-4\geq k$. Since $|S|\leq k$ it cannot violate the cut-condition.
\item[Case 3] If $1\leq |S\cap K_{k-1}|\leq k-3$, $d(S)\geq 2k-6\geq k$ even by considering the edges leaving $S$ within $K_{k-1}$.
\item [Case 4] IF $S\subset K_{1,k}$, $S$ must contain the center of the star else it trivially holds the condition. Observe that each non-central vertex of $K_{1,k}$ has an edge leaving $S$ toward $K_{k-1}$ and so does at least one edge of the star. It completes the proof.
\end{description}

It has been known for some time, that not only linkedness and weak-linkedness force high-connectivity and edge-connectivity of the graph, but that sufficiently large connectivity and edge-connectivity imply linkedness and weak-linkedness, respectively. It would be interesting to see if similar result can be proved about the relation of the cut-conditions and path-pairability.

\subsection*{Path-pairability of hypercubes and grids}

As discussed previously, path-pairable graphs on $n$ vertives have a certain lower bound of approximately $O(\frac{\log n}{\log\log n})$ on the minimal value of the maximum degree $\Delta$. On the other hand, the smallest achieved maximum degree provided by Theorem \ref{m} has order of magnitude $O(\sqrt{n})$, still leaving plenty of room for improvements on both sides. One particularly interesting and promising path-pairable  candidate is the $d$-dimensional hypercube $Q_d$ on $n=2^d$ vertices with $\Delta(Q_d)=d=\log n$. Although it is known that $Q_n$ is not path-pairable for even values of  $d$ (\cite{F}), the question is open for odd dimensional hypercubes if $d\geq 5$ ($Q_1$ and $Q_3$ are both path-pairable). 
\begin{conjecture}[\cite{Cs}]\label{q}
The $(2k+1)$-dimensional hypercube $Q_{2k+1}$ is path-pairable for all $k\in\mathbb{N}$.
\end{conjecture}

The question regarding the path-pairability number of larger $n$ dimensional affine and projective grids, that is, the Cartesian product of $d$ paths or $d$ cycles has not been answered either. It can be derived rather easily from Theorem \ref{pp oom}, that sufficiently large $d$-dimensional projective grids are $O(2^d)$-path-pairable. Similar result concerning affine grids can be obtained. We leave the proof of both statements to the reader. On the other hand, it can be proved that a $d$ dimensional projective grid is at most $O((2d)^{2d})$-path-pairable, regardless of its size, if the grid is large enough in every dimension.
\begin{proposition}
Let $G=C_{m_1}\square C_{m_2}\square\dots C_{m_d}$, where $C_{m_i}$ denotes a cycle of length $m_i$ and  $m_i\geq (2d+1)$, $i=1,2,\dots, d$. Then $pp(G)\leq (2d)^{d-1}\cdot(2d+1) $ 
\end{proposition}
\begin{proof}
Way may assume $|G| \geq 2\cdot(2d)^{d-1}\cdot(2d+1)$, else the statement is trivial. Consider now the $d$-dimensional subgrid $G_0=C_{2d}\square\dots\square C_{2d}\square C_{2d+1}$. Easy to see that
$G_0$ violates the cut-condition as $V(G_0)=(2d)^{d-1}\cdot(2d+1)>2\cdot( (d-1)(2d)^{d-2}(2d+1)+ (2d)^{d-1})=d(G_0)$. It shows that $G$ is less than $(2d)^{d-1}\cdot(2d+1)$-path-pairable.
\end{proof}

The presented bounds are still far apart and leave plenty of room for improvements.

\begin{question}
Determine the values of $pp(P_{m_1}\square P_{m_2}\square\dots P_{m_d})$ and $pp(C_{m_1}\square C_{m_2}\square\dots C_{m_d})$ ($P_{m_i}$ denotes a path of length $m_i$).
\end{question}

\subsection*{Possible extension of Theorem \ref{m}}
This paper only deals with a special type of products of complete bipartite graphs. With a detailed and cumbersome analysis of our presented techniques, one can prove that the product graph $K_{a,b}\square K_{c,d}$ is path-pairable if $\frac{max(a,b,c,d)}{min(a,b,c,d)}<2$ and $a,b,c,d$ are large enough (in terms of the previous ratio).
 We close up with highlighting, that path-pairability of $K_{a,b}\square K_{c,d}$ in the general case is still subject to further investigation, as well as proposing another intriguing open question motivated by \cite{grid}. 
 \begin{question}
For which values of $a,b,c,d\in\mathbb{Z}^+$ ($a\leq b$, $c\leq d$) is the product graph $K_{a,b}\square K_{c,d}$ path-pairable?
\end{question}
\begin{question}
What are the necessary and sufficient conditions for a graph $G$, such that $G\square K_n$ will be path-pairable if $n$ is large enough?
\end{question}

\section*{Acknowledgement}

The author wishes to thank Professor Ervin Gy\H ori  and Professor Ralph Faudree for their helpful suggestions, interest, and guidance.

Research was supported by the Balassi Institute, the Fulbright Commission, and the Rosztoczy Foundation.

\end{document}